\documentclass[11pt,twoside]{article}
\usepackage{a4wide}
\usepackage[final]{showkeys}
\usepackage{fancyhdr}
\usepackage{graphicx}
\usepackage{amsmath}
\usepackage{amsfonts,amssymb,amsbsy}
\usepackage{upgreek}
\usepackage{appendix}
\usepackage{color}

\newcommand{\D}{\displaystyle}
\newcommand{\R}{{\mathbb R}}   
\newcommand{\N}{{\mathbb N}}   

\newtheorem{theorem}{Theorem}

\newtheorem{lemma}[theorem]{Lemma}

\newtheorem{remark}[theorem]{Remark}
\newenvironment{proof}[1][Proof]{\textbf{#1.} }{\hfill \raisebox{-0.1em}{$\Box$}\\}

\begin{document}
\markboth{P.~Skrzypacz and D.~Wei}{On the solvability of the Brinkman-Forchheimer-extended Darcy equation} 

\title{ON THE SOLVABILITY OF THE BRINKMAN-FORCHHEIMER-EXTENDED DARCY EQUATION}

\author{
  Piotr Skrzypacz\footnote{
  Dr. Piotr Skrzypacz, School of Science and Technology,
  Nazarbayev University,
  53 Kabanbay Batyr Ave., Astana 010000 Kazakhstan,
  Email: {\it piotr.skrzypacz@nu.edu.kz}
  } ~and~ Dongming Wei\footnote{
  Dr. Dongming Wei, School of Science and Technology,
  Nazarbayev University,
  53 Kabanbay Batyr Ave., Astana 010000 Kazakhstan,
  Email: {\it dongming.wei@nu.edu.kz}
  }~
}

\maketitle

\begin{abstract}
The nonlinear Brinkman-Forchheimer-extended Darcy equation is used
to model some porous medium flow in chemical reactors of packed bed type. 
The results concerning the existence and uniqueness of a weak solution 
are presented for nonlinear convective  flows in medium with nonconstant porosity and for small data. 
Furthermore, the finite element  approximations to the flow profiles  in the fixed bed reactor are presented for several Reynolds numbers
at the non-Darcy's range.
\end{abstract}

\textbf{2010 Mathematics Subject Classification (MSC):}~
76D03,  
35Q35  

\textbf{Keywords:}~
Brinkman-Forchheimer Equation, Packed Bed Reactors, Existence and Uniqueness of Solution\\

\section{{I}ntroduction}
In this section we introduce the mathematical model describing incompressible isothermal flow in porous medium without reaction. The
considered equations for the velocity and pressure fields are for flows in fluid saturated porous media. Most of
research results for flows in porous media are based on  the Darcy equation which is considered to be  a suitable model at a small range of Reynolds numbers.  However, there are restrictions of Darcy equation for modeling some porous medium flows, e.g. in  closely packed
medium, saturated fluid flows at slow velocity but with relatively large Reynolds numbers. The flows in such closely packed medium behave nonlinearly and can not be modelled accurately by the Darcy equation which is linear. The deficiency can be circumvented with the Brinkman--Forchheimer-extended Darcy law for flows in closely packed media, which leads to the following model:
Let $\Omega\subset\R^n$, $n=2, 3$, represent the reactor channel. We denote
its boundary by $\Gamma=\partial\Omega$. The conservation of volume-averaged
values of momentum and mass in the packed reactor reads as follows
\begin{equation}\label{s2eq1}
\begin{array}{rcr}
\displaystyle
-\textrm{div}\,\left(\varepsilon \nu \nabla \boldsymbol u-\varepsilon\boldsymbol
 u\otimes\boldsymbol u \right)+\frac{\varepsilon}{\varrho}\nabla p + \sigma(\boldsymbol u)=\boldsymbol f&\textrm{in}&\Omega\,,\\
\textrm{div}\,(\varepsilon\boldsymbol u)=0&\textrm{in}&\Omega\,,
\end{array}
\end{equation}
where $\boldsymbol u\,:\Omega\to\R^n$,~ $p\,:\Omega\to\R$ denote the unknown velocity and pressure, respectively. The
positive quantity $\varepsilon=\varepsilon(\boldsymbol x)$ stands for porosity which describes the proportion of the
non-solid volume to the total volume of material and varies spatially in general. The expression
$\sigma(\boldsymbol{u})$ represents the friction forces caused by the packing and will be specified later on. The
right-hand side $\boldsymbol f$ represents an outer force (e.g. gravitation), $\varrho$ the constant fluid density and
$\nu$ the constant kinematic viscosity of the fluid, respectively. The expression $\boldsymbol u\otimes\boldsymbol u$ symbolizes the dyadic product of $\boldsymbol u$ with itself. 

The formula given by Ergun \cite{Ergun_1} will be used to model the influence
of the packing on the flow inertia effects
\begin{align}\label{s2eq2}
\sigma(\boldsymbol u)=150\nu\frac{(1-\varepsilon)^2}{\varepsilon^2d_p^2}\boldsymbol u
+1.75\frac{1-\varepsilon}{\varepsilon d_p}\boldsymbol u|\boldsymbol u|\;.
\end{align}
Thereby $d_p$ stands for the diameter of pellets and $|\cdot|$ denotes the Euclidean vector norm. The linear term in (\ref{s2eq2}) accounts for the head loss according to Darcy and  the quadratic term according to Forchheimer law, respectively. For the derivation of the equations, modelling and homogenization questions in porous media we refer to e.g. \cite{Bey_1, Hornung_1}. 
To close the system (\ref{s2eq1}) we prescribe  Dirichlet boundary condition
\begin{equation}\label{s2eq3}
\boldsymbol u\arrowvert_{\Gamma}=\boldsymbol g\,,
\end{equation}
whereby
\begin{equation}\label{zerocomp}
\int\limits_{\Gamma_i}\varepsilon\boldsymbol
g\cdot\boldsymbol n\,ds=0 
\end{equation}
has to be fulfilled on each connected component $\Gamma_i$ of the
boundary $\Gamma$.
We remark
 that in the case of polygonally bounded domain the outer normal
 vector $\boldsymbol n$ has jumps and thus the above integral
 should be replaced by a sum of integrals over each side of
 $\Gamma$. 
The distribution of porosity $\varepsilon$ is assumed to satisfy the following bounds
\begin{equation}\label{A1}
0<\varepsilon_0\le \varepsilon(\boldsymbol x)\le
\varepsilon_1\le 1\quad\forall\,\boldsymbol x\in\Omega\tag{\textrm{A1}}\,,
\end{equation}
with some constants $0<\varepsilon_0,\;\varepsilon_1\le 1$. 

A comprehemsive account of fluid flows through porous media beyond the Darcy law's valid regimes and classified by the Reynolds number,  can be found in, e.g., \cite{Zhao_1}. Also, see \cite{Upton_1} for simulating pumped water levels in abstraction boreholes  using such nonlinear Darcy-Forchheimer law, and \cite{Grillo_1}, \cite{Sobieski_1}, and \cite{Lal_1} for recent referenes on this model.

In the next section we use the porosity distribution which is estimated for packed beds consisting of spherical particles and takes the near wall channelling effect into account. This kind of porosity distribution obeys assumption \eqref{A1}.

Let us introduce dimensionless quantities 
\begin{align*}
\boldsymbol u^*=\frac{\boldsymbol u}{U_0}\,,\quad
p^*=\frac{p}{\varrho U_0^2}\,,\quad 
\boldsymbol{x}^*=\frac{\boldsymbol x}{d_p}\,,\quad
\boldsymbol{g}^*=\frac{\boldsymbol g}{U_0}\,,
\end{align*}
whereby $U_0$ denotes the magnitude of some reference velocity. For simplicity of notation we omit the asterisks. Then, the reactor flow problem reads in dimensionless form as follows
\begin{equation}\label{s2eq4}
\left\{\begin{array}{rclrl}
\displaystyle
-\textrm{div}\,\left(\frac{\varepsilon}{Re}\nabla \boldsymbol u-
\varepsilon\boldsymbol u\otimes\boldsymbol u \right)+\varepsilon\nabla p + 
\frac{\alpha}{Re}\boldsymbol u+\beta\boldsymbol u|\boldsymbol u|&=&\boldsymbol f &\textrm{in} &\Omega\,,\\
\textrm{div}\,(\varepsilon\boldsymbol u)&=&0 &\textrm{in} &\Omega\,,\\
\boldsymbol u&=&\boldsymbol g &\textrm{on} &\Gamma\,,
\end{array}\right.
\end{equation}
where 
\begin{equation}\label{alpha_beta_model}
\begin{split}
\alpha(\boldsymbol x)&=150\kappa^2(\boldsymbol x)\,,\qquad \beta(\boldsymbol x)=1.75\kappa(\boldsymbol x)
\end{split}
\end{equation} 
with 
\begin{equation}\label{kappa_model}
\kappa(\boldsymbol x)=\frac{1-\varepsilon(\boldsymbol x)}{\varepsilon(\boldsymbol x)}\,,
\end{equation}
and the Reynolds number is defined by 
$$Re=\frac{U_{0}\,d_p}{\nu}\,.$$ 

The existence and uniqueness of solution of the nonlinear model (\ref{s2eq4}) with constant porosity and without the convective term has been established in \cite{Kaloni}. We will extend this result to the case when the porosity depends on the location and with the convective term in this work.
\begin{remark}\label{remarkNSE_1}
\eqref{s2eq4} becomes a Navier-Stokes problem if $\varepsilon\equiv 1$.
\end{remark}
{\bf Notation}~~Throughout the work we use the following notations for function
spaces. For $m\in\N_0$, $p\ge 1$ and bounded subdomain $G\subset\Omega$ let $W^{m,p}(G)$ be the
usual Sobolev space equipped with norm $\|\cdot\|_{m,p,G}$. If $p=2$, we denote the Sobolev space 
by $H^m(G)$ and use the standard abbreviations $\|\cdot\|_{m,G}$ and $|\cdot|_{m,G}$ for the norm 
and seminorm, respectively. We denote by $D(G)$ the space of $C^\infty(G)$ functions with compact 
support contained in $G$. Furthermore,  $H_0^m(G)$ stands for the closure of $D(G)$ with respect
to the norm $\|\cdot\|_{m,G}$. The counterparts spaces consisting of vector valued functions 
will be denoted by bold faced symbols like $\boldsymbol{H}^m(G):=[H^m(G)]^n$ or
$\boldsymbol{D}(G):=[D(G)]^n$. The $L^2$ inner product over $G\subset\Omega$ and
$\partial G\subset\partial\Omega$ will be denoted by $(\cdot,\cdot)_G$ and
$\langle\cdot,\cdot\rangle_{\partial G}$, respectively. In the case $G=\Omega$ the
domain index will be omitted. In the following we denote by $C$ the generic
constant which is usually independent of the model parameters, otherwise
dependences will be indicated.
%
\section{Existence and uniqueness results}\label{s3}
In the following the porosity $\varepsilon$ is assumed to belong to
$W^{1,3}(\Omega)\cap L^\infty(\Omega)$. We start with the weak formulation of problem \eqref{s2eq4} and look for its solution in suitable Sobolev spaces.
\subsection{Variational formulation}
Let $$L^2_0(\Omega):=\{ v\in L^2(\Omega): (v,1)=0\}$$ be the space
consisting of $L^2$ functions with zero mean value. We define the spaces
\begin{equation*}
\boldsymbol{X}:=\boldsymbol{H}^1(\Omega)\,,\quad
\boldsymbol{X}_0:=\boldsymbol{H}^1_0(\Omega)\,,\quad
Q:=L^2(\Omega)\,,\quad M:=L^2_0(\Omega)\,,
\end{equation*}
and
\begin{equation*}
\boldsymbol{V}:=\boldsymbol{X}_0\times M\,.
\end{equation*} 
Let us introduce the following bilinear forms 
\begin{equation*}
\begin{alignedat}{3}
&a:\,\boldsymbol{X}\times\boldsymbol{X}&\to\R\,,&&\qquad a(\boldsymbol u, \boldsymbol v)&=\frac{1}{Re}\bigl(\varepsilon\nabla\boldsymbol u,\nabla\boldsymbol v\bigr)\,,\\[1.5ex]
&b:\,\boldsymbol{X}\times Q&\to\R\,,&&\qquad b(\boldsymbol u,q)&=\bigl(\textrm{div}(\varepsilon\boldsymbol{u}),q\bigr)\,,\\[1.5ex]
&c:\,\boldsymbol{X}\times \boldsymbol{X}&\to\R\,,&&\qquad c(\boldsymbol u,\boldsymbol v)&=\frac{1}{Re}\bigl(\alpha\boldsymbol{u},\boldsymbol{v}\bigr)\,.\\
\end{alignedat}
\end{equation*}
Furthermore, we define the semilinear form
\begin{equation*}
d:\,\boldsymbol{X}\times\boldsymbol{X}\times\boldsymbol{X}\rightarrow\R\,,\qquad d(\boldsymbol w;\boldsymbol u,\boldsymbol v)=\bigl(\beta |\boldsymbol{w}|\boldsymbol{u},\boldsymbol{v}\bigr)\,,
\end{equation*}
and trilinear form
\begin{equation*}
n:\,\boldsymbol{X}\times\boldsymbol{X}\times\boldsymbol{X}\rightarrow\R\,,\qquad
n(\boldsymbol w,\boldsymbol u,\boldsymbol v)=\bigl((\varepsilon \boldsymbol w\cdot\nabla)\boldsymbol u,\boldsymbol v\bigr)\,.
\end{equation*}
We set
\begin{equation*}
A(\boldsymbol{w};\boldsymbol{u},\boldsymbol{v}):=a(\boldsymbol{u},\boldsymbol{v})+c(\boldsymbol{u},\boldsymbol{v})+n(\boldsymbol{w},\boldsymbol{u},\boldsymbol{v})+d(\boldsymbol{w};\boldsymbol{u},\boldsymbol{v})\,.
\end{equation*}
Multiplying momentum and mass balances in (\ref{s2eq4}) by test functions $\boldsymbol v\in\boldsymbol{X}_0$ and $q\in M$, respectively, and integrating by parts implies the weak formulation:\\[2ex]
\hspace*{1cm}Find $\displaystyle (\boldsymbol u,p)\in \boldsymbol{X}\times M$ with $\boldsymbol u\arrowvert_{\Gamma}=\boldsymbol g$\quad such that
\begin{equation}\label{s3eq5}
A(\boldsymbol{u};\boldsymbol{u},\boldsymbol{v})-b(\boldsymbol{v},p)+b(\boldsymbol{u},q)=(\boldsymbol{f},\boldsymbol{v})\quad\forall\;(\boldsymbol v,q)\in \boldsymbol{V}\,.
\end{equation}
First, we recall the following result from \cite{BernardiLaval}:
\begin{theorem}\label{s3thm1}
The mapping $u\mapsto \varepsilon u$ is an isomorphism from
$H^1(\Omega)$ onto itself and from $H^1_0(\Omega)$ onto
itself. It holds for all $u\in H^1(\Omega)$
\begin{equation*}
\|\varepsilon u\|_1\le
C\{\varepsilon_1+|\varepsilon|_{1,3}\}\,\|u\|_1\qquad\text{and}\qquad
\left\|\frac{u}{\varepsilon}\right\|_1\le
C\left\{\varepsilon_0^{-1}+\varepsilon_0^{-2}\,|\varepsilon|_{1,3}\right\}\|u\|_1\,.
\end{equation*}
\end{theorem} 
In the following the closed subspace of $\boldsymbol H^1_0(\Omega)$ defined by
\begin{equation*}
\boldsymbol{W}=\{\boldsymbol w\in\boldsymbol H^1_0(\Omega):\quad b(\boldsymbol w,q)=0\quad\forall\; q\in L^2_0(\Omega)\}.
\end{equation*}
will be employed. Next, we establish and prove some properties of trilinear form $n(\cdot,\cdot,\cdot)$ and nonlinear form $d(\cdot;\cdot,\cdot)$.
\begin{lemma}\label{s3lem2}
Let $\boldsymbol u, \boldsymbol v\in\boldsymbol H^1(\Omega)$ and $\boldsymbol w\in\boldsymbol H^1(\Omega)$ with $\text{div}\,(\varepsilon\boldsymbol w)=0$ and $\boldsymbol w\cdot\boldsymbol n\arrowvert_\Gamma=0$. Then we have  
\begin{equation}\label{s3eq6}
n(\boldsymbol w,\boldsymbol u,\boldsymbol v)=-n(\boldsymbol w,\boldsymbol v,\boldsymbol u)\,.
\end{equation}
Furthermore, the trilinear form $n(\cdot,\cdot,\cdot)$ and the nonlinear form
$d(\cdot;\cdot,\cdot)$ are continuous, i.e.
\begin{equation}\label{s3eq7}
|n(\boldsymbol u,\boldsymbol v,\boldsymbol w)|\le C_\varepsilon\,\|\boldsymbol u\|_1 \|\boldsymbol v\|_1 \|\boldsymbol w\|_1\quad\forall\;\boldsymbol u,\boldsymbol v,\boldsymbol w\in \boldsymbol H^1(\Omega)\,,
\end{equation}
\begin{equation}\label{s3eq7b}
|d(\boldsymbol u,\boldsymbol v,\boldsymbol w)|\le C_\varepsilon\, \|\boldsymbol u\|_1 \|\boldsymbol v\|_1 \|\boldsymbol w\|_1\quad\forall\;\boldsymbol u,\boldsymbol v,\boldsymbol w\in \boldsymbol H^1(\Omega)\,,
\end{equation} 
and for $\boldsymbol{u}\in\boldsymbol{W}$ and for a sequence $\boldsymbol u^k\in\boldsymbol{W}$ with $\lim\limits_{k\to\infty}\|\boldsymbol u^k-\boldsymbol{u}\|_0=0$, we have also
\begin{equation}\label{s3eq8}
\lim\limits_{k\to\infty}n(\boldsymbol u^k,\boldsymbol u^k,\boldsymbol v)=n(\boldsymbol u,\boldsymbol u,\boldsymbol v)\quad\forall\;\boldsymbol v\in\boldsymbol{W}.
\end{equation}
\end{lemma}
\begin{proof}
We follow the proof of \cite[Lemma 2.1, \S 2, Chapter IV]{giro} and adapt it to the trilinear form 
$$n(\boldsymbol w,\boldsymbol u,\boldsymbol v)=\bigl((\varepsilon \boldsymbol
w\cdot\nabla)\boldsymbol u,\boldsymbol v\bigr)=\sum\limits_{i,j=1}^n\bigl(\varepsilon
w_j\partial_ju_i,v_i\bigr)\,,$$ which has the weighting factor $\varepsilon$. Hereby, symbols with subscripts denote components of bold faced vectors, e.g. $\boldsymbol{u}=(u_i)_{i=1,\ldots,n}$. Let $\boldsymbol{u}\in\boldsymbol{H}^1$, $\boldsymbol{v}\in\boldsymbol{D}(\Omega)$ and $\boldsymbol w\in\boldsymbol{W}$. 
Integrating by parts and employing density argument, we obtain immediately \eqref{s3eq6}
\begin{equation*}
\begin{split}
&\sum\limits_{i,j=1}^n\bigl(\varepsilon w_j\partial_j u_i,v_i\bigr)
=-\sum\limits_{i,j=1}^n\bigl(\partial_j\left(\varepsilon w_jv_i\right),u_i \bigr)+\sum\limits_{i,j=1}^n\langle\varepsilon w_j n_j u_i,v_i\rangle\\
&=-\sum\limits_{i,j=1}^n\bigl(\varepsilon w_j\partial_jv_i,u_i\bigr)-\bigl(\text{div}\,(\varepsilon\boldsymbol w)\boldsymbol u,\boldsymbol v\bigr)
+\bigl\langle (\varepsilon\boldsymbol{w}\cdot\boldsymbol{n})\boldsymbol{u},\boldsymbol{v}\bigr\rangle\\
&=-n(\boldsymbol w,\boldsymbol v,\boldsymbol u).
\end{split}
\end{equation*}
From Sobolev embedding $H^1(\Omega)\hookrightarrow L^4(\Omega)$ (see \cite{Adams}) and H\"older inequality follows
\begin{equation*}
\left|\bigl(\varepsilon w_j\partial_ju_i,v_i\bigr)\right|\le |\varepsilon|_{0,\infty}\,\|w_j\|_{0,4}\,\|\partial_ju_i\|_0\,\|v_i\|_{0,4}
\le C\,|\varepsilon|_{0,\infty}\,\| w_j\|_1\, |u_i|_1\,\|v_i\|_1\,,
\end{equation*} 
and consequently the proof of \eqref{s3eq7} is completed.
Since $\lim\limits_{k\rightarrow\infty}\|u_{i}^ku_{j}^k-u_iu_j\|_{0,1}=0$ and $\displaystyle\varepsilon\partial_jv_i\in L^\infty(\Omega)$, the continuity estimate \eqref{s3eq7} implies
\begin{equation*}
\begin{split}
\lim\limits_{k\rightarrow\infty} n(\boldsymbol u^k,\boldsymbol u^k,\boldsymbol v)&=-\lim\limits_{k\rightarrow\infty} n(\boldsymbol u^k,\boldsymbol v,\boldsymbol u^k)=-\lim\limits_{k\to\infty}\sum\limits_{i,j=1}^n\bigl(\varepsilon u_j^k\,\partial_jv_i^k,u_i^k\bigr)\\
&=-\sum\limits_{i,j=1}^n\bigl(\varepsilon u_j\partial_j v_i,u_i\bigr)=-n(\boldsymbol u,\boldsymbol v,\boldsymbol u)=n(\boldsymbol u,\boldsymbol u,\boldsymbol v)\,.
\end{split}
\end{equation*}
The continuity of $d(\cdot;\cdot,\cdot)$ follows from H\"older inequality and Sobolev embedding $H^1(\Omega)\hookrightarrow L^4(\Omega)$ (see \cite{Adams})
\begin{equation*}
|d(\boldsymbol u;\boldsymbol v,\boldsymbol w)|\le |\beta|_\infty\, \|\boldsymbol u\|_{0,4}\, \|\boldsymbol v\|_{0,4}\, \|\boldsymbol w\|_0\le C_\varepsilon  \|\boldsymbol u\|_1\, \|\boldsymbol v\|_1\, \|\boldsymbol w\|_1\,.
\end{equation*} 
\end{proof}
In the next stage we consider the  difficulties caused by prescribing the inhomogeneous Dirichlet
boundary condition. Analogous difficulties are already encountered in the analysis of
Navier--Stokes problem. We will carry out the study of three dimensional case. The extension in
two dimensions can be constructed analogously. Since $\boldsymbol g\in\boldsymbol{H}^{1/2}(\Gamma)$, we can extend $\boldsymbol g$ inside of $\Omega$ in the form of 
$$\boldsymbol g=\varepsilon^{-1}\,\textrm{curl}\,\boldsymbol h$$ 
with some $\boldsymbol h\in \boldsymbol{H}^2(\Omega)$. The operator $\textrm{curl}$ is
defined then as 
$$
\textrm{curl}\, \boldsymbol{h}=(\partial_2 h_3-\partial_3 h_2,\, \partial_3
h_1-\partial_1 h_3,\, \partial_1 h_2-\partial_2 h_1)\,.
$$
We note that in the two dimensional
case the vector potential $\boldsymbol{h}\in \boldsymbol{H}^2(\Omega)$ can be replaced
by a scalar function $h\in
H^2(\Omega)$ and the operator $\textrm{curl}$ is then redefined as $\textrm{curl}\, h=(\partial_2
h, -\partial_1 h)$. Our aim is to adapt the extension of Hopf (see \cite{Hopf}) to our
model. We recall that for any parameter $\mu>0$ there exists a scalar  function $\varphi_\mu\in C^2(\bar{\Omega})$ such that
\begin{equation}\label{Ex}
\left.\begin{split}
&\hspace{-0.5cm}\bullet\quad\varphi_\mu=1~\text{in some neighborhood of
$\Gamma$ (depending on $\mu$)}\,,\\[2ex]
&\hspace{-0.5cm}\bullet\quad\varphi_\mu(\boldsymbol{x})=0~\text{if
$d_\Gamma(\boldsymbol{x})\ge 2\exp{(-1/\mu)}$\,, where $d_\Gamma(\boldsymbol
x):=\inf\limits_{\boldsymbol
y\in\Gamma}|\boldsymbol{x}-\boldsymbol{y}|$}\\[-0.5ex] 
&\hspace{0.5cm}\text{denotes the distance of $\boldsymbol x$ to
$\Gamma$}\,,\\[2ex]
&\hspace{-0.5cm}\bullet\quad
|\partial_j\varphi_\mu(\boldsymbol{x})|\le\mu/d_\Gamma(\boldsymbol{x})~~
\text{if~~$d_\Gamma(\boldsymbol x)< 2\exp{(-1/\mu)}$\,, $j=1,\ldots, n\,.$}
\end{split}\;\right\}\tag{\text{Ex}}
\end{equation}
For the construction of $\varphi_\mu$ see also \cite[Lemma 2.4,
\S 2, Chapter IV]{giro}.\\[1.5ex]
Let us define 
\begin{equation}\label{s3extension}
\boldsymbol g_\mu:=\varepsilon^{-1}\,\textrm{curl}\,(\varphi_\mu\boldsymbol h)\,. 
\end{equation}
In the following lemma we establish bounds which are crucial for proving existence of velocity.
\begin{lemma}\label{s3lem3}
The function $\boldsymbol g_\mu$ satisfies the following conditions
\begin{equation}\label{s3eq9}
\textrm{div}\,(\varepsilon\boldsymbol g_\mu) =0,\quad \boldsymbol
g_\mu\arrowvert_{\Gamma}=\boldsymbol g\qquad\forall\,\mu>0\,,
\end{equation}   
and for any $\delta>0$ there exists sufficiently small $\mu>0$ such that
\begin{align}
\label{s3eq10} |d(\boldsymbol u+\boldsymbol g_\mu;\boldsymbol
g_\mu,\boldsymbol u)| & \le \delta\,\|\beta\|_{0,\infty}\,|\boldsymbol u|_1\bigl(|\boldsymbol u|_1+\|\boldsymbol g_\mu\|_0\bigr)\qquad\forall\;\boldsymbol u\in\boldsymbol{X}_0\,,\\
\label{s3eq11} |n(\boldsymbol u,\boldsymbol g_\mu,\boldsymbol u)| &\le \delta\,|\boldsymbol u|_1^2\qquad\forall\;\boldsymbol u\in\boldsymbol{W}\,. 
\end{align} 
\end{lemma}
\begin{proof}
The relations in \eqref{s3eq9} are obvious. We follow \cite{Kaloni} in order to show \eqref{s3eq10}. Since $\boldsymbol h\in\boldsymbol{H}^2(\Omega)$ Sobolev's embedding theorem implies $\boldsymbol h\in \boldsymbol{L}^\infty(\Omega)$, so we get according to the properties of $\varphi_\mu$ in \eqref{Ex} the following bound 
\begin{equation*}
|\boldsymbol g_\mu|\le C\,\varepsilon_0^{-1}\,\left\{|\nabla\boldsymbol
h|+\frac{\mu}{d_\Gamma(\boldsymbol{x})} |\boldsymbol h|\right\}
\le C\,\left\{\frac{\mu}{d_\Gamma(\boldsymbol{x})}+|\nabla\boldsymbol h|\right\}.\end{equation*}
Defining 
$$
\Omega_\mu:=\{\boldsymbol x\in\Omega:\;d_\Gamma(\boldsymbol{x})<2\exp(-1/\mu)\}
$$ 
we obtain from Cauchy-Schwarz and triangle inequalities 
\begin{equation}\label{s3eq12}
\begin{split}
|\bigl(\beta |\boldsymbol u+\boldsymbol g_\mu|, \boldsymbol g_\mu\cdot\boldsymbol u\bigr)|
&\le\,\|\beta\|_{0,\infty}\,\|\boldsymbol u\|_0\,\|\boldsymbol
u\cdot\boldsymbol g_\mu\|_{0,\Omega_\mu}\\
&\qquad+\|\beta\|_{0,\infty}\,\|\boldsymbol g_\mu\|_0\,\|\boldsymbol u\cdot\boldsymbol g_\mu\|_{0,\Omega_\mu}\,,
\end{split}
\end{equation}
\begin{equation*}
\begin{split}
&\|\boldsymbol u\cdot\boldsymbol g_\mu\|_{0,\Omega_\mu}^2
\le \int\limits_{\Omega_\mu}|\boldsymbol u|^2|\boldsymbol g_\mu|^2 d\boldsymbol x\\ 
&\le C\int\limits_{\Omega_\mu}|\boldsymbol
u|^2\biggl\{\bigl(\mu/d_\Gamma(\boldsymbol{x})\bigr)^2+2\mu/d_\Gamma(\boldsymbol{x})\,|\nabla\boldsymbol h|+|\nabla\boldsymbol h|^2\biggr\}d\boldsymbol x\\
&\le C\left \{\mu^2 \|\boldsymbol u/d_\Gamma\|_{0,\Omega_\mu}^2+2\mu
\|\boldsymbol u/d_\Gamma\|_{0,\Omega_\mu}\,\|\boldsymbol{u}\|_{0,4,\Omega_\mu}\,\bigl\||\nabla\boldsymbol h|\bigr\|_{0,4,\Omega_\mu}
+\|\boldsymbol u\|_{0,4,\Omega_\mu}^2\bigl\||\nabla\boldsymbol h|\bigr\|_{0,4,\Omega_\mu}^2\right\}\\
&\le C\left \{\mu \|\boldsymbol u/d_\Gamma\|_{0,\Omega_\mu}+\|\boldsymbol u\|_{0,4}\bigl\||\nabla\boldsymbol h|\bigr\|_{0,4,\Omega_\mu}\right\}^2,
\end{split}
\end{equation*}
and consequently
\begin{equation}\label{s3eq13}
\|\boldsymbol u\cdot\boldsymbol g_\mu\|_{0,\Omega_\mu}\le C\left\{\mu
\|\boldsymbol u/d_\Gamma\|_{0,\Omega_\mu}+\|\boldsymbol u\|_{0,4}\bigl\||\nabla\boldsymbol h|\bigr\|_{0,4,\Omega_\mu}\right\}\,.
\end{equation}
Applying Hardy inequality (see \cite{Adams})
$$\|v/d_\Gamma\|_0\le C|v|_1\quad\forall\;v\in H_0^1(\Omega)$$
and using Sobolev embedding $H^1(\Omega)\hookrightarrow L^4(\Omega)$, estimate  (\ref{s3eq13}) becomes 
\begin{equation}\label{s3eq14}
\|\boldsymbol u\cdot\boldsymbol g_\mu\|_{0,\Omega_\mu}\le C\lambda(\mu)\|\boldsymbol u\|_1, 
\end{equation}
where 
$$\lambda(\mu):=\max\bigl\{\mu,\bigl\||\nabla\boldsymbol
h|\bigr\|_{0,4,\Omega_\mu}\bigr\}\,.$$ 
From \eqref{s3eq12}, \eqref{s3eq14}, Poincar\'e inequality and  from  the fact that $\lim\limits_{\mu\to 0}\lambda(\mu) = 0$  we conclude that for any $\delta>0$ we can choose sufficiently small $\mu>0$ such that  
\begin{equation*}
|(\beta\,|\boldsymbol u+\boldsymbol g_\mu|\boldsymbol g_\mu,\boldsymbol u)| 
\le\delta\,\|\beta\|_{0,\infty}\,|\boldsymbol{u}|_1\bigl(|\boldsymbol{u}|_1+\|\boldsymbol{g}_\mu\|_0\bigr)
\end{equation*}
holds. Therefore the proof of estimate \eqref{s3eq10} is completed. Now, we take a look at the trilinear convective term
\begin{equation*}
\begin{split}
n(\boldsymbol{u},\boldsymbol{g}_\mu,\boldsymbol{u})&=
\bigl((\varepsilon\boldsymbol{u}\cdot\nabla)\boldsymbol{g}_\mu, \boldsymbol{u}\bigr)_{\Omega_\mu}
=\biggl((\varepsilon\boldsymbol{u}\cdot\nabla)\left\{\varepsilon^{-1}\,\text{curl}\,(\varphi_\mu\boldsymbol{h})\right\}, \boldsymbol{u}\biggr)_{\Omega_\mu}\\
&=\biggl((\boldsymbol{u}\cdot\nabla)\left\{\text{curl}\,(\varphi_\mu\boldsymbol{h})\right\}, \boldsymbol{u}\biggr)_{\Omega_\mu}-\bigl((\boldsymbol{u}\cdot\nabla\varepsilon)\,\boldsymbol{g}_\mu,\boldsymbol{u}\bigr)_{\Omega_\mu}\,.
\end{split}
\end{equation*}
The first term of above difference becomes small due to
\cite[Lemma 2.3, \S 2, Chapter IV]{giro}, and it satisfies
\begin{equation}\label{s3eq15}
\left|\bigl((\boldsymbol{u}\cdot\nabla)\left\{\text{curl}\,(\varphi_\mu\boldsymbol{h})\right\}, \boldsymbol{u}\bigr)_{\Omega_\mu}\right|=\left|\bigl((\boldsymbol{u}\cdot\nabla)(\varepsilon\boldsymbol{g}_\mu), \boldsymbol{u}\bigr)_{\Omega_\mu}\right|\le \delta |\boldsymbol{u}|_1^2
\end{equation}
as long as $\mu>0$ is chosen sufficiently small. Using H\"older inequality, Sobolev embedding $H^1(\Omega)\hookrightarrow L^6(\Omega)$ yields
\begin{equation*}
\left|\bigl((\boldsymbol{u}\cdot\nabla\varepsilon)\,\boldsymbol{g}_\mu,\boldsymbol{u}\bigr)_{\Omega_\mu}\right|\le C\|\varepsilon\|_{1,3}\,\|\boldsymbol{g}_\mu\cdot\boldsymbol{u}\|_0\,\|\boldsymbol{u}\|_1\,,
\end{equation*}
which together with \eqref{s3eq14} implies for sufficiently small $\mu>0$ the bound
\begin{equation}\label{s3eq16}
\left|\bigl((\boldsymbol{u}\cdot\nabla\varepsilon)\,\boldsymbol{g}_\mu,\boldsymbol{u}\bigr)_{\Omega_\mu}\right|\le \delta |\boldsymbol{u}|_1^2\,.
\end{equation}
From \eqref{s3eq15} and \eqref{s3eq16} follows the desired estimate \eqref{s3eq11}.
\end{proof}
While the general framework for linear and non-symmetric saddle point problems can be found  in \cite{BernardiLaval},
our problem requires more attention due to its nonlinear character. Setting $\boldsymbol{w}:=\boldsymbol u-\boldsymbol g_\mu$, the weak formulation \eqref{s3eq5} is equivalent to the following problem\\[2ex]
\hspace*{0.4cm}Find $\displaystyle (\boldsymbol w,p)\in \boldsymbol{V}$ such that
\begin{gather}\label{s3eq17}
\begin{split}
A(\boldsymbol{w}+\boldsymbol{g}_\mu;\boldsymbol{w}+\boldsymbol{g}_\mu,\boldsymbol{v})-b(\boldsymbol{v},p)+b(\boldsymbol{w}+\boldsymbol{g}_\mu,q)=(\boldsymbol{f},\boldsymbol{v})\quad\forall\; (\boldsymbol{v},q)\in\boldsymbol{V}\,.
\end{split}
\end{gather}
Let us define the nonlinear mapping $G:\; \boldsymbol{W}\rightarrow \boldsymbol{W}$ with
\begin{equation}\label{s3eq18}
\begin{split}
\bigl[G(\boldsymbol w),\boldsymbol v\bigr]:=&a(\boldsymbol w+\boldsymbol g_\mu,\boldsymbol v)+c(\boldsymbol w+\boldsymbol g_\mu,\boldsymbol v)-(\boldsymbol f,\boldsymbol v)\\
&\;+n(\boldsymbol w+\boldsymbol g_\mu,\boldsymbol w+\boldsymbol g_\mu,\boldsymbol v)+d(\boldsymbol w+\boldsymbol g_\mu;\boldsymbol w+\boldsymbol g_\mu,\boldsymbol v)\,,
\end{split}
\end{equation}
whereby $[\cdot,\cdot]$ defines the inner product in $\boldsymbol{W}$ via
$[u,v]:=(\nabla u,\nabla v)$.   
Then, the variational problem (\ref{s3eq17}) reads in the space of $\varepsilon$-weighted divergence free functions $\boldsymbol{W}$ as follows\\[2ex]
\hspace*{0.4cm}Find $\displaystyle \boldsymbol w\in \boldsymbol{W}$ such that
\begin{equation}\label{s3eq19}
\bigl[G(\boldsymbol w),\boldsymbol v\bigr]=0\quad\forall\;\boldsymbol v\in\boldsymbol{W}.
\end{equation}
\subsection{Solvability of nonlinear saddle point problem}
We start our study of the nonlinear operator problem \eqref{s3eq19} with the following lemma.
\begin{lemma}\label{s3lem4}
The mapping $G$ defined in (\ref{s3eq18}) is continuous and there exists $r>0$ such that
\begin{equation}\label{s3eq20}
\bigl[G(\boldsymbol u),\boldsymbol u\bigr] >0\quad\forall\;\boldsymbol u\in\boldsymbol{W}\quad\textrm{with}\quad |\boldsymbol u|_1=r.
\end{equation}
\end{lemma}
\begin{proof}
Let $(\boldsymbol u^k)_{k\in\N}$ be a sequence in $\boldsymbol{W}$ with $\lim\limits_{k\to\infty}\|\boldsymbol u^k-\boldsymbol u\|_1=0$. Then, applying Cauchy--Schwarz inequality and \eqref{s3eq11}, we obtain for any $\boldsymbol v\in\boldsymbol{W}$
\begin{equation*}
\begin{split}
&\left|\bigl[G(\boldsymbol u^k)-G(\boldsymbol u),\boldsymbol v\bigr]\right|\le \frac{1}{Re}\left|\bigl(\varepsilon\nabla(\boldsymbol u^k-\boldsymbol u),\nabla\boldsymbol v\bigr)\right|
+\frac{1}{Re}\left|\bigl(\alpha(\boldsymbol u^k-\boldsymbol u),\boldsymbol
v\bigr)\right|\\
&\quad +\left|\bigl(\beta|\boldsymbol u^k+\boldsymbol g_\mu|(\boldsymbol
u^k-\boldsymbol u),\boldsymbol v\bigr)\right|+\left|\bigl(\beta (|\boldsymbol
u^k+\boldsymbol g_\mu|-|\boldsymbol u+\boldsymbol g_\mu|)(\boldsymbol
u+\boldsymbol g_\mu),\boldsymbol v\bigr)\right|\\
&\quad+\left|n(\boldsymbol u^k,\boldsymbol u^k, \boldsymbol v)-n(\boldsymbol
u,\boldsymbol u,\boldsymbol v)\right|
+\left|n(\boldsymbol u^k-\boldsymbol u,\boldsymbol g_\mu,\boldsymbol
v)\right|+\left|n(\boldsymbol g_\mu, \boldsymbol u^k-\boldsymbol
u,,\boldsymbol v)\right|\\
&\le\frac{\varepsilon_1}{Re}|\boldsymbol u^k-\boldsymbol u|_1|\boldsymbol v|_1
+\frac{1}{Re}\|\alpha\|_{0,\infty}\|\boldsymbol u^k-\boldsymbol u\|_0\|\boldsymbol v\|_0\\
&\quad+\|\beta\|_{0,\infty}\|\boldsymbol u^k+\boldsymbol g_\mu\|_{0,4}\|\boldsymbol u^k-\boldsymbol u\|_0\|\boldsymbol v\|_{0,4}
+\|\beta\|_{0,\infty}\|\boldsymbol u+\boldsymbol g_\mu\|_{0,4}\|\boldsymbol u^k-\boldsymbol u\|_0\|\boldsymbol v\|_{0,4}\\
&\quad+\left|n(\boldsymbol u^k,\boldsymbol u^k, \boldsymbol v)-n(\boldsymbol
u,\boldsymbol u,\boldsymbol v)\right|
+C\|\boldsymbol u^k-\boldsymbol u\|_1\|\boldsymbol g_\mu\|_1\|\boldsymbol
v\|_1\,.
\end{split}
\end{equation*}
The boundedness of $\boldsymbol u^k$ in $\boldsymbol{W}$,  \eqref{s3eq8}, the Poincar\'e inequality, and the above inequality  imply that
$$
\left|\bigl[G(\boldsymbol u^k)-G(\boldsymbol u),\boldsymbol v\bigr]\right|\to
0\quad\text{as}\quad k\to\infty\qquad\forall\,
\boldsymbol{v}\in\boldsymbol{W}\,.
$$
Thus, employing 
$$
|G(\boldsymbol u^k)-G(\boldsymbol
u)|_1=\sup\limits_{\overset{\boldsymbol
v\in\boldsymbol{W}}{\boldsymbol v\neq \boldsymbol
0}}\frac{\bigl[G(\boldsymbol u^k)-G(\boldsymbol u),\boldsymbol
v\bigr]}{|\boldsymbol v|_1}\,,
$$
we state that $G$ is continuous. Now, we note that for any $\boldsymbol u\in\boldsymbol{W}$  we have
\begin{equation}\label{s3eq21}
\begin{split}
&\bigl[G(\boldsymbol u),\boldsymbol u\bigr]
=\frac{1}{Re}\bigl(\varepsilon\nabla(\boldsymbol u+\boldsymbol g_\mu),\nabla\boldsymbol u\bigr)+\frac{1}{Re}\bigl(\alpha(\boldsymbol u+\boldsymbol g_\mu),\boldsymbol u\bigr)\\
&\quad +\bigl(\beta|\boldsymbol u+\boldsymbol g_\mu|(\boldsymbol u+\boldsymbol g_\mu),\boldsymbol u\bigr)
+n(\boldsymbol u+\boldsymbol g_\mu,\boldsymbol u+\boldsymbol g_\mu,\boldsymbol u)-(\boldsymbol f,\boldsymbol u)\\
&\ge\frac{\varepsilon_0}{Re}|\boldsymbol u|_1^2-\frac{\varepsilon_1}{Re}|(\nabla\boldsymbol g_\mu,\nabla\boldsymbol u)|
+\frac{1}{Re}(\alpha\boldsymbol u,\boldsymbol u)-\frac{1}{Re}|(\alpha\boldsymbol g_\mu,\boldsymbol u)|\\
&\quad +(\beta|\boldsymbol u+\boldsymbol g_\mu|,|\boldsymbol
u|^2)-\left|(\beta |\boldsymbol u+\boldsymbol g_\mu|\boldsymbol
g_\mu,\boldsymbol u)\right|\\
&\quad +n(\boldsymbol u,\boldsymbol g_\mu,\boldsymbol u)+n(\boldsymbol g_\mu,\boldsymbol g_\mu,\boldsymbol u)-\|\boldsymbol f\|_0\|\boldsymbol u\|_0\\
&\ge\frac{\varepsilon_0}{Re}|\boldsymbol u|_1^2-\frac{\varepsilon_1}{Re}|\boldsymbol g_\mu|_1|\boldsymbol u|_1\\
&\quad -\frac{1}{Re}\|\alpha\|_{0,\infty}\|\boldsymbol g_\mu\|_0\|\boldsymbol u\|_0
-\left|(\beta |\boldsymbol u+\boldsymbol g_\mu|\boldsymbol g_\mu,\boldsymbol
u)\right|\\
&\quad -\left|n(\boldsymbol u,\boldsymbol g_\mu,\boldsymbol u)\right|
-C\|\boldsymbol g_\mu\|_1^2\|\boldsymbol u\|_1
-\|\boldsymbol f\|_0\|\boldsymbol u\|_0\,.
\end{split}
\end{equation}
From the Poincar\'e  inequality, we infer the estimate
\begin{equation*}
\|v\|_1\le C|v|_1\quad\forall\; v\in H_0^1(\Omega),
\end{equation*}
which together with \eqref{s3eq10}, \eqref{s3eq11} and \eqref{s3eq21} results in 
\begin{equation*}
\begin{split}
&\bigl[G(\boldsymbol u),\boldsymbol
u\bigr]\ge\left\{\frac{\varepsilon_0}{Re}-\delta(1+\|\beta\|_{0,\infty})\right\}|\boldsymbol u|_1^2\\
&\quad -\Bigl\{\frac{\varepsilon_1}{Re}|\boldsymbol
g_\mu|_1+C_1\frac{1}{Re}\|\alpha\|_{0,\infty}\|\boldsymbol g_\mu\|_0
+\delta \|\beta\|_{0,\infty}\|\boldsymbol g_\mu\|_0
+C_2\|\boldsymbol g_\mu\|_1^2+C_3\|\boldsymbol f\|_0\Bigr\}|\boldsymbol u|_1.
\end{split}
\end{equation*}
Choosing $\delta$ such that
\begin{equation*}
0<\delta<\delta_0:=\frac{\varepsilon_0}{Re}\bigl(1+\|\beta\|_{0,\infty}\bigr)^{-1}\,,
\end{equation*}
and $r>r_0$ with
\begin{equation}\label{s3eq22}
\begin{split}
r_0:=\frac{\D\frac{\varepsilon_1}{Re}|\boldsymbol
g_\mu|_1+\frac{1}{Re}C_1\|\alpha\|_{0,\infty}\|\boldsymbol
g_\mu\|_0+\delta \|\beta\|_{0,\infty}\|\boldsymbol g_\mu\|_0+C_2\|\boldsymbol g_\mu\|_1^2+C_3\|\boldsymbol f\|_0}
{\D\frac{\varepsilon_0}{Re}-\delta(1+\|\beta\|_{0,\infty})}\,,
\end{split}
\end{equation}
leads to the desired assertion (\ref{s3eq20}).
\end{proof}
The following lemma plays a key role in the existence proof.
\begin{lemma}\label{s3lem5}
Let $Y$ be finite-dimensional Hilbert space with inner product $[\cdot,\cdot]$ inducing a norm $\|\cdot\|$, and $T:\,Y\rightarrow Y$ be a continuous mapping such that 
$$\bigl[T(x),x\bigr]>0\quad\textrm{for}\quad\|x\|=r_0>0.$$
Then there exists $x\in Y$, with $\|x\|\le r_0$, such that
$$T(x)=0.$$ 
\end{lemma}
\begin{proof}
See \cite{Lions_1}.
\end{proof}
Now we are able to prove the main result concerning existence of velocity.
\begin{theorem}\label{s3thm6}
The problem (\ref{s3eq19}) has at least one solution $\boldsymbol u\in\boldsymbol{W}$.
\end{theorem}
\begin{proof}
We construct the approximate sequence of Galerkin solutions. Since the space $\boldsymbol{W}$ is separable, there exists a sequence of linearly independent elements $\left(\boldsymbol w^i\right)_{i\in\N}\subset\boldsymbol{W}$. Let $\boldsymbol{X}_m$ be the finite dimensional subspace of $\boldsymbol{W}$ with 
$$\boldsymbol{X}_m:=\text{span}\{\boldsymbol{w}^i\,,~ i=1,\ldots ,m\}$$ 
and endowed with the scalar product of $\boldsymbol{W}$. Let $\boldsymbol u^m=\sum\limits_{j=1}^ma_j\boldsymbol w^j,\;
a_j\in\R$\,, be a Galerkin solution of (\ref{s3eq19}) defined by 
\begin{align}
\label{s3eq23} & \bigl[ G(\boldsymbol u^m),\boldsymbol w^j\bigr]=0,\quad\forall\; j=1,\ldots ,m\,.
\end{align}
From Lemma \ref{s3lem4}  and Lemma \ref{s3lem5} we conclude that  
\begin{equation}\label{s3eq24}
\bigl[G(\boldsymbol u^m),\boldsymbol w\bigr]=0\quad\forall\;\boldsymbol w\in\boldsymbol X_m
\end{equation}
has a solution  $\boldsymbol u^m\in\boldsymbol X_m$. The unknown coefficients $a_j$ can be obtained from the algebraic system (\ref{s3eq23}). On the other hand, multiplying (\ref{s3eq23}) by $a_j$, and adding the equations for $j=1,\ldots, m$ we have
\begin{equation*}
\begin{split}
0&=\bigl[G(\boldsymbol u^m),\boldsymbol u^m\bigr]\\
&\ge \left\{\frac{1}{Re}-\delta(1+\|\beta\|_{0,\infty})\right\}|\boldsymbol u^m|_1^2\\
&\quad -\Bigl\{\frac{1}{Re}|\boldsymbol
g_\mu|_1+C_1\frac{1}{Re}\|\alpha\|_{0,\infty}\|\boldsymbol g_\mu\|_0
+\delta \|\beta\|_{0,\infty}\|\boldsymbol g_\mu\|_0
+C_2\|\boldsymbol g_\mu\|_1^2+C_3\|\boldsymbol f\|_0\Bigr\}|\boldsymbol u^m|_1.
\end{split}
\end{equation*}
This gives together with (\ref{s3eq22}) the uniform boundedness in $\boldsymbol{W}$
\begin{equation*}
|\boldsymbol u^m|_1\le r_0,
\end{equation*}
therefore there exists $\boldsymbol u\in\boldsymbol{W}$ and a subsequence $m_k\rightarrow\infty$ ( we write for the convenience $m$ instead of $m_k$ ) such that
\begin{equation*}
\boldsymbol u^m\rightharpoonup\boldsymbol u\quad\textrm{in}\quad\boldsymbol{W}.
\end{equation*}
Furthermore, the compactness of embedding $H^1(\Omega)\hookrightarrow L^4(\Omega)$ implies 
\begin{equation*}
\boldsymbol u^m\rightarrow\boldsymbol u\quad\textrm{in}\quad\boldsymbol{L}^4(\Omega).
\end{equation*}
Taking the limit in (\ref{s3eq24}) with $m\rightarrow\infty$ we get
\begin{equation}\label{s3eq25}
\bigl[G(\boldsymbol u),\boldsymbol w\bigr]=0\quad\forall\;\boldsymbol w\in\boldsymbol X_m.
\end{equation}
Finally, we apply the continuity argument and state that (\ref{s3eq25}) is preserved for any $\boldsymbol w\in\boldsymbol{W}$, therefore $\boldsymbol u$ is the solution of (\ref{s3eq19}).
\end{proof}
For the reconstruction of the pressure we need inf-sup-theorem 
\begin{theorem}\label{s3thm7}
Assume that the bilinear form $b(\cdot,\cdot)$ satisfies the inf-sup condition
\begin{equation}\label{s3eq26}
\inf\limits_{q\in M}\sup\limits_{\boldsymbol v\in\boldsymbol{X}_0}\frac{b(\boldsymbol v,q)}{|\boldsymbol v|_1\,\|q\|_0}\ge\gamma>0.
\end{equation}
Then, for each solution $\boldsymbol u$ of the nonlinear problem (\ref{s3eq19}) there exists a unique pressure $p\in M$ such that the pair $(\boldsymbol u,p)\in\boldsymbol{V}$ is a solution of the homogeneous problem (\ref{s3eq17}).
\end{theorem}
\begin{proof}
See \cite[Theorem 1.4, \S 1, Chapter IV]{giro}. 
\end{proof}
We end up this subsection by proving the existence of the pressure. 
\begin{theorem}\label{s3thm8}
Let $\boldsymbol w$ be solution of problem \eqref{s3eq19}. Then, there exists unique pressure $p\in M$.
\end{theorem}
\begin{proof}
We verify the inf-sup condition (\ref{s3eq26}) of Theorem
\ref{s3thm7} by employing the isomorphism of Theorem \ref{s3thm1}.
From \cite[Corollary 2.4, \S 2, Chapter I]{giro} follows that for any $q$ in $L_0^2(\Omega)$ there exists $\boldsymbol v$ in $\boldsymbol H_0^1(\Omega)$ such that 
$$(\textrm{div}\,\boldsymbol v,q)\ge \gamma^*\|\boldsymbol v\|_1\|q\|_0$$
with a positive constant $\gamma^*$. Setting $\boldsymbol u=\boldsymbol v/\varepsilon$ and applying the isomorphism in Theorem \ref{s3thm1}, we obtain the estimate
\begin{equation*}
b(\boldsymbol u,q)=(\textrm{div}\,\boldsymbol v,q)\ge
\gamma^*\|\boldsymbol v\|_1\|q\|_0\ge \gamma_\varepsilon
\|\boldsymbol u\|_1\|q\|_0
\end{equation*}
where
$\D\gamma_\varepsilon=\frac{\gamma^*}{C\left\{\varepsilon_0^{-1}+\varepsilon_0^{-2}\,|\varepsilon|_{1,3}\right\}}$. From the above estimate we conclude the inf-sup condition \eqref{s3eq26}.
\end{proof} 
\subsection{Uniqueness of weak solution}
We exploit a priori estimates in order to prove uniqueness of weak velocity and pressure.
\begin{theorem}\label{s3thm9}
If $\|\boldsymbol g_\mu\|_1$, $\displaystyle\|\boldsymbol f\|_{-1}:=\sup\limits_{\boldsymbol 0\neq \boldsymbol v\in \boldsymbol H^1(\Omega)}\frac{(\boldsymbol f,\boldsymbol v)}{\|\boldsymbol{v}\|_1}$ are sufficiently small, then the solution of (\ref{s3eq19}) is unique. 
\end{theorem}
\begin{proof}
Assume that $(\boldsymbol u_1, p_1)$ and $(\boldsymbol u_2, p_2)$ are two different solutions of (\ref{s3eq17}). From \eqref{s3eq6} in Lemma \ref{s3lem2} we obtain $n(\boldsymbol w,\boldsymbol u,\boldsymbol u)=0~~\forall\;\boldsymbol w,\boldsymbol u\in\boldsymbol{W}$. Then, we obtain
\begin{equation}\label{s3eq27}
\begin{split}
0&=\bigl[G(\boldsymbol u_1)-G(\boldsymbol u_2),\boldsymbol
u_1-\boldsymbol u_2\bigr]\\
&=a(\boldsymbol u_1-\boldsymbol u_2,\boldsymbol u_1-\boldsymbol u_2)+c(\boldsymbol u_1-\boldsymbol u_2,\boldsymbol u_1-\boldsymbol u_2)-(\boldsymbol f,\boldsymbol u_1-\boldsymbol u_2)\\
&\quad + n(\boldsymbol u_1+\boldsymbol g_\mu,\boldsymbol u_1+\boldsymbol g_\mu,\boldsymbol u_1-\boldsymbol u_2)-n(\boldsymbol u_2+\boldsymbol g_\mu,\boldsymbol u_2+\boldsymbol g_\mu,\boldsymbol u_1-\boldsymbol u_2)\\
&\quad + (\beta |\boldsymbol u_1+\boldsymbol g_\mu|(\boldsymbol u_1+\boldsymbol g_\mu),\boldsymbol u_1-\boldsymbol u_2)\\
&\quad-\bigl(\beta |\boldsymbol u_2+\boldsymbol g_\mu|(\boldsymbol u_2+\boldsymbol g_\mu),\boldsymbol u_1-\boldsymbol u_2)\\
&\ge \frac{\varepsilon_0}{Re}|\boldsymbol u_1-\boldsymbol u_2|_1^2-\|\boldsymbol f\|_{-1}\|\boldsymbol u_1-\boldsymbol u_2\|_1\\
&\quad +n(\boldsymbol u_1-\boldsymbol u_2,\boldsymbol u_2+\boldsymbol g_\mu,\boldsymbol u_1-\boldsymbol u_2)\\
&\quad +\bigl(\beta|\boldsymbol u_1+\boldsymbol g_\mu|(\boldsymbol u_1-\boldsymbol u_2),\boldsymbol u_1-\boldsymbol u_2\bigr)\\
&\quad +\bigl(\beta(|\boldsymbol u_1+\boldsymbol g_\mu|-|\boldsymbol u_2+\boldsymbol g_\mu|)(\boldsymbol{u}_2+\boldsymbol{g}_\mu),\boldsymbol u_1-\boldsymbol u_2\bigr)\\
&\ge \frac{\varepsilon_0}{Re}|\boldsymbol u_1-\boldsymbol u_2|_1^2-\|\boldsymbol f\|_{-1}\|\boldsymbol u_1-\boldsymbol u_2\|_1\\
&\quad -\left|n(\boldsymbol u_1-\boldsymbol u_2,\boldsymbol
u_2,\boldsymbol u_1-\boldsymbol u_2)\right|
-\left|n(\boldsymbol
u_1-\boldsymbol u_2,\boldsymbol g_\mu,\boldsymbol u_1-\boldsymbol
u_2)\right|\\
&\quad -\|\beta\|_{0,\infty}\left|\bigl(|\boldsymbol u_1+\boldsymbol
g_\mu|\cdot |\boldsymbol u_1-\boldsymbol u_2|,|\boldsymbol u_1-\boldsymbol
u_2|\bigr)\right|\\
&\quad -\|\beta\|_{0,\infty}\left|\bigl(\bigl||\boldsymbol u_1+\boldsymbol
g_\mu|-|\boldsymbol u_2+\boldsymbol
g_\mu|\bigr|\cdot |\boldsymbol{u}_2+\boldsymbol{g}_\mu|,|\boldsymbol
u_1-\boldsymbol u_2|\bigr)\right|\,.
\end{split}
\end{equation}
From Cauchy-Schwarz inequality and Sobolev embedding $H^1(\Omega)\hookrightarrow
L^4(\Omega)$ we deduce
\begin{equation}\label{s3eq28}
\bigl|\bigl(|\boldsymbol u_1+\boldsymbol g_\mu|\cdot |\boldsymbol u_1-\boldsymbol
u_2|,|\boldsymbol u_1-\boldsymbol u_2|\bigr)\bigr|\le C\left\{\|\boldsymbol u_1\|_0+\|\boldsymbol g_\mu\|_0\right\}\|\boldsymbol u_1-\boldsymbol u_2\|_1^2\,,
\end{equation}
\begin{equation}\label{s3eq29}
\begin{split}
&\bigl|\bigl(\bigl||\boldsymbol u_1+\boldsymbol g_\mu|-|\boldsymbol
u_2+\boldsymbol
g_\mu|\bigr|\cdot |\boldsymbol{u}_2+\boldsymbol{g}_\mu|,|\boldsymbol
u_1-\boldsymbol u_2|\bigr)\bigr|\\
&\le C\left\{\|\boldsymbol u_2\|_0+\|\boldsymbol g_\mu\|_0\right\}\|\boldsymbol u_1-\boldsymbol u_2\|_1^2, 
\end{split}
\end{equation}
and according to (\ref{s3eq7}) we have
\begin{equation}\label{s3eq30}
|n(\boldsymbol u_1-\boldsymbol u_2,\boldsymbol u_2,\boldsymbol u_1-\boldsymbol u_2)|\le C\|\boldsymbol u_2\|_1\|\boldsymbol u_1-\boldsymbol u_2\|_1^2,
\end{equation}
and by (\ref{s3eq9}) we can find $\mu$ such that
\begin{equation}\label{s3eq31}
|n(\boldsymbol u_1-\boldsymbol u_2,\boldsymbol g_\mu,\boldsymbol u_1-\boldsymbol u_2)|\le \frac{\varepsilon_0}{4 Re}\|\boldsymbol u_1-\boldsymbol u_2\|_1^2.
\end{equation}
Now, we find upper bounds for $\boldsymbol u_1$ and $\boldsymbol u_2$. Testing the equation (\ref{s3eq17}) with $\boldsymbol u$ results in
\begin{equation*}
\begin{split}
\frac{\varepsilon_0}{Re}\|\boldsymbol u\|_1^2 &\le  \|\boldsymbol f\|_{-1}\|\boldsymbol u\|_1+\frac{\varepsilon_0}{Re}\|\boldsymbol g_\mu\|_1\|\boldsymbol u\|_1+C\|\boldsymbol g_\mu\|_0\|\boldsymbol u\|_0\\
&\quad +C\|\boldsymbol g_\mu\|_1^2\|\boldsymbol u\|_1
+C\|\beta\|_{0,\infty}\|\boldsymbol g_\mu\|_0\|\boldsymbol
u\|_1^2+C\|\beta\|_{0,\infty}\|\boldsymbol g_\mu\|_{0,4}^2\|\boldsymbol u\|_1\,.
\end{split}
\end{equation*}
From Sobolev embedding  $H^1(\Omega)\hookrightarrow L^4(\Omega)$  we deduce for sufficiently small $\|\boldsymbol g_\mu\|_1$
\begin{equation}\label{s3eq32}
\|\boldsymbol u\|_1\le\frac{\|\boldsymbol f\|_{-1}+C_1\|\boldsymbol
g_\mu\|_1+C_2\|\boldsymbol
g_\mu\|_1^2}{\displaystyle\frac{\varepsilon_0}{Re}-C_3\|\beta\|_{0,\infty}\|\boldsymbol g_\mu\|_1}=:C\bigl(\|\boldsymbol g_\mu\|_1, \|\boldsymbol f\|_{-1}\bigr).
\end{equation}
Putting (\ref{s3eq28})-(\ref{s3eq32}) into (\ref{s3eq27}) and using the inequality 
$$\|\boldsymbol f\|_{-1}\|\boldsymbol u_1-\boldsymbol u_2\|_1\le \frac{\varepsilon_0}{4Re}\|\boldsymbol u_1-\boldsymbol u_2\|_1^2+\frac{2 Re}{\varepsilon_0}\|\boldsymbol f\|_{-1}^2$$ 
we obtain
\begin{equation}\label{s3eq33}
\begin{split}
0&\ge \frac{\varepsilon_0}{2 Re}\|\boldsymbol u_1-\boldsymbol
u_2\|_1^2-\frac{2Re}{\varepsilon_0}\|\boldsymbol
f\|_{-1}^2-C\bigl(\|\boldsymbol g_\mu\|_1,\|\boldsymbol
f\|_{-1}\bigr)\,\|\beta\|_{0,\infty}\|\boldsymbol u_1-\boldsymbol u_2\|_1^2\\
&\quad -\frac{\varepsilon_0}{4 Re}\|\boldsymbol u_1-\boldsymbol u_2\|_1^2-C\bigl(\|\boldsymbol g_\mu\|_1,\|\boldsymbol f\|_{-1}\bigr)\|\boldsymbol u_1-\boldsymbol u_2\|_1^2\,.
\end{split}
\end{equation}
For sufficiently small $\|\boldsymbol g_\mu\|_1$, $ \|\boldsymbol f\|_{-1}$ the constant $C(\|\boldsymbol g_\mu\|_1,\|\boldsymbol f\|_{-1})$ in \eqref{s3eq32} gets small and consequently the right hand side of (\ref{s3eq33}) is nonnegative. This implies
$\boldsymbol u_1=\boldsymbol u_2$ and according to Theorem \ref{s3thm8} is $p_1-p_2=0$.
\end{proof}
\section{A Channel Flow Problem in Packed Bed Reactors}
In this section, we provide an example of the flow problem in packed bed reactors with numerical solutions at small and relatively large Reynolds numbers to show the nonlinear behavior of the velocity solutions.
Let the reactor channel be represented by $\Omega=(0,L)\times (-R,R)$ where
$R=5$ and $L=60$. 
\begin{figure}[hbt!]
\begin{center}
\includegraphics[scale=0.95]{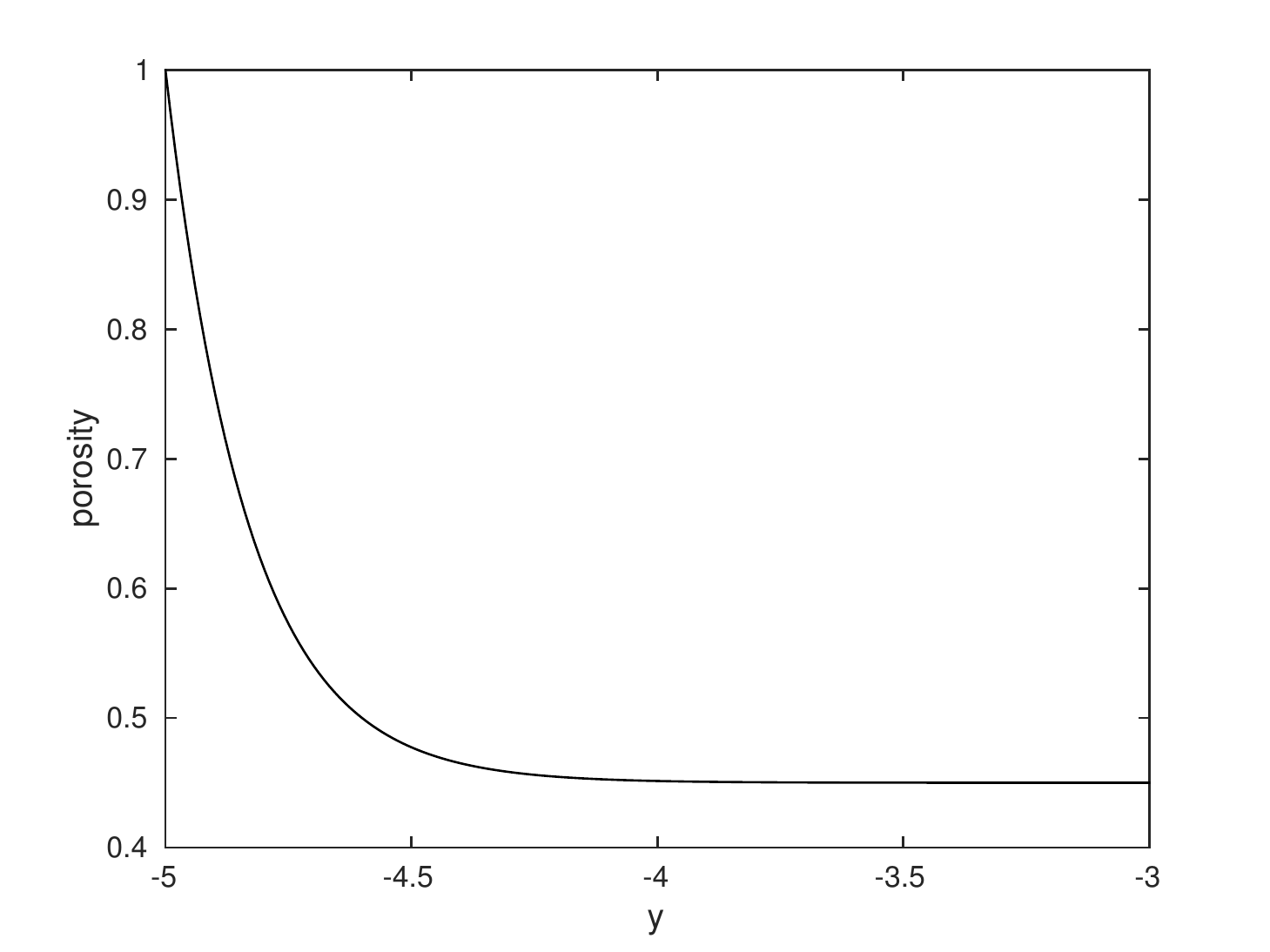}
\end{center}
\caption{Varying porosity.\label{fig_poro}}
\end{figure}
In all computations we use the porosity distribution which is determined experimentally and takes into account the effect of wall channelling in packed bed reactors
\begin{equation} 
\varepsilon(x,y)=\varepsilon(y)=\varepsilon_\infty\left\{1+\frac{1-\varepsilon_\infty}{\varepsilon_\infty}\,e^{-6(R-|y|)}\right\}\,,
\end{equation} 
where $\varepsilon_\infty=0.45$. The distribution of the porosity is presented in Figure~\ref{fig_poro}. 
We  distinguish between the inlet, outlet and membrane parts of domain boundary $\Gamma$, and denote them by $\Gamma_{in}$, $\Gamma_{out}$ and $\Gamma_w$, respectively. Let 
\begin{equation*}
\begin{array}{lcl}
\Gamma_{in}&=&\{(x,y)\in\Gamma:\; x=0\}\,,\\
\Gamma_{out}&=&\{(x,y)\in\Gamma:\; x=L\}\,,\\
\Gamma_{w}&=&\{(x,y)\in\Gamma:\; y=-R,\; y=R\}\,.
\end{array}
\end{equation*}
At the inlet $\Gamma_{in}$ and at the membrane wall we prescribe Dirichlet boundary conditions, namely the plug flow conditions 
\begin{equation*}
\boldsymbol{u}\arrowvert_{\Gamma_{in}}=\boldsymbol{u}_{in}=(u_{in},0)^T\,,
\end{equation*}
and
\begin{equation*}
\boldsymbol{u}\arrowvert_{\Gamma_{w}}=\boldsymbol{u}_{w}=
\begin{cases} 
(0,u_w)^T & \quad\text{for}\quad y=-R\,,\\
(0,-u_w)^T & \quad\text{for}\quad y=R\,,
\end{cases}
\end{equation*}
whereby $u_{in}>0$, $u_w>0$. At the outlet $\Gamma_{out}$ we set the following outflow boundary condition
\begin{equation*}
-\frac{1}{Re}\,\frac{\partial\boldsymbol{u}}{\partial\boldsymbol{n}}+p\boldsymbol{n}=\boldsymbol{0}
\end{equation*}
where $\boldsymbol{n}$ denotes the outer normal. In order to avoid
discontinuity between the inflow and wall conditions we replace constant
profile by trapezoidal one with zero value at the corners. Our
computations are carried out on the Cartesian mesh using biquadratic conforming and discontinuous piecewise linear finite elements for the approximation of the velocity and pressure, respectively. The finite element analysis of the Brinkman-Forchheimer-extended Darcy equation will be conducted in the forthcoming work. The plots of velocity magnitude in fixed bed reactor ($u_w=0$) are
 presented along the vertical axis $x=50$. In the investigated reactor the
 inlet velocity is assumed to be normalized ($u_{in}=1$). Due to the
 variation of porosity we might expect higher velocity at the reactor
 walls $\Gamma_w$. This tunnelling effect can be well observed in Figure~\ref{fig4} which shows the velocity profiles
 for different Reynolds numbers. We remark that the maximum of velocity magnitude decreases with increasing Reynolds numbers.   
\begin{figure}[hbt!]
\begin{center}
\includegraphics[scale=0.55]{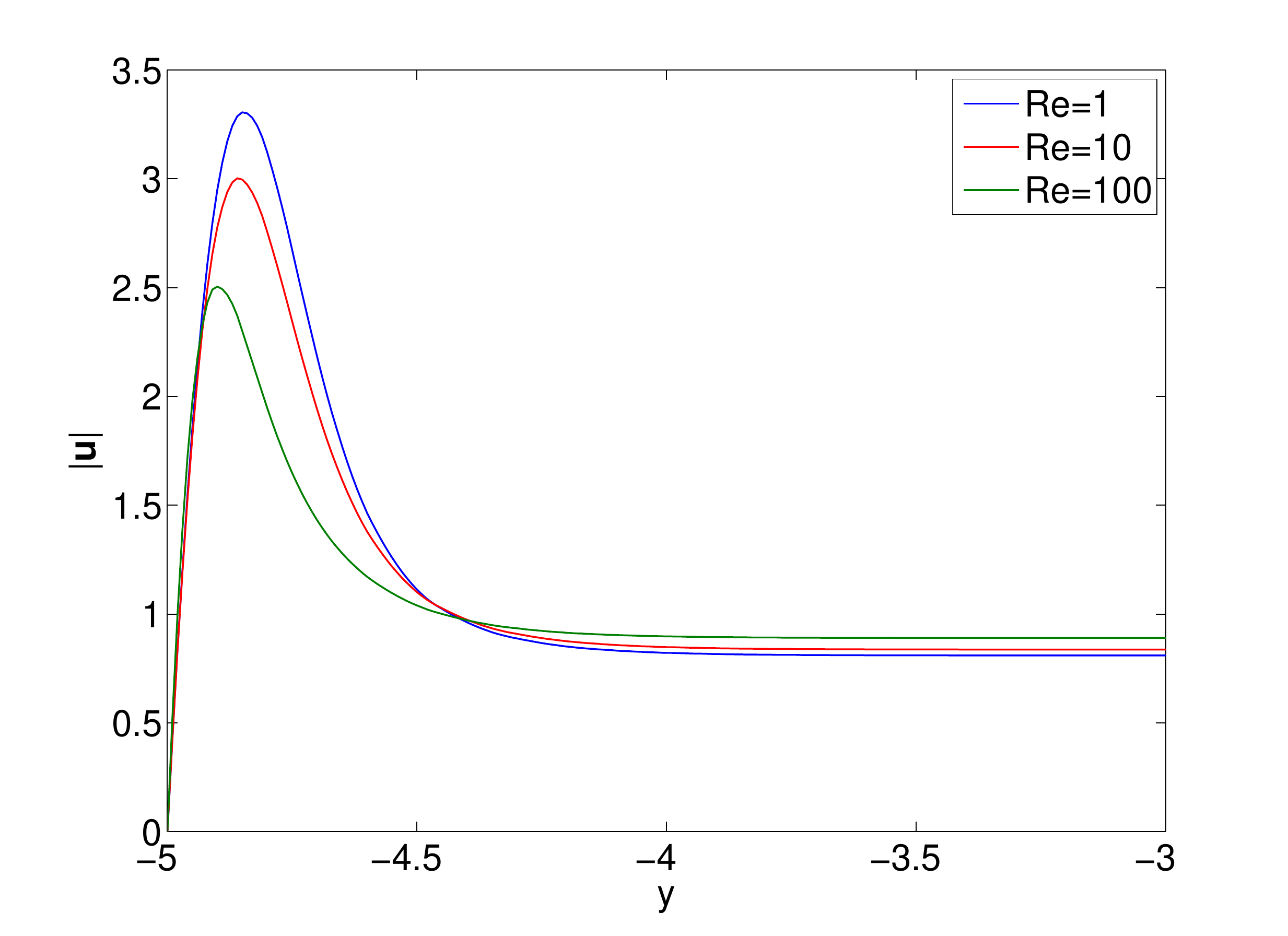}
\end{center}
\caption{Flow profiles in fixed bed reactor at $x=50$.\label{fig4}}
\end{figure}
\section{ Conclusion}
In this work, we have extended the existence and uniqueness of solution result in literature for the porous medium flow problem based on the nonlinear Brinkman-Forchheimer-extended Darcy law. The existing result is valid only for constant porosity and without the considered convection effects, and our result holds for variable porosity and it includes convective effects. We also provided a numerical solution to demonstrate the nonlinear velocity solutions at moderately large Reynolds numbers for which case the Brinkman-Forchheimer-extended Darcy law applies.


\end{document}